\documentclass[a4,11pt]{article}
\usepackage{amsmath,amsfonts,amsthm,latexsym,amssymb,mathrsfs}

\newtheorem{df}{Definition}[section]
\newtheorem{thm}[df]{Theorem}
\newtheorem{pro}[df]{Proposition}

\newtheorem{rema}[df] {Remark}
\newtheorem{lem}[df] {Lemma}
\def\sfstp{{\hskip-1em}{\bf.}{\hskip1em}}

\def\subject#1{\renewcommand{\thefootnote}{}
\footnote{\noindent 2010 \it Mathematics Subject Classification\rm . Primary 46H05; Secondary 15A09.\par
 \hskip.2truecm \it Key words and phrases. \rm (weighted) Moore-Penrose inverse,
(weighted) EP element, group inverse, Banach algebra, Banach space operator.\par
{#1}}}

\title{ \bf  WEIGHTED MOORE-PENROSE INVERTIBLE \\ AND WEIGHTED EP BANACH ALGEBRA ELEMENTS \/}

\author { \normalsize ENRICO BOASSO,  DRAGAN S.  DJORDJEVI\'C AND DIJANA  MOSI\'C}

\date{   }

\begin{document}
\maketitle \thispagestyle{empty} 

\subject{  }

\vskip 1truecm

\setlength{\baselineskip}{12pt}

\begin{abstract} \noindent The weighted Moore-Penrose inverse will be introduced and studied in the
context of Banach algebras. In addition, weighted EP Banach algebra elements
will be characterized. The Banach space operator case will be also considered. 
\end{abstract}

\section{\sfstp Introduction}\setcounter{df}{0}
\
\indent The Moore-Penrose inverse in Banach algebras was
introduced by V. Rako\v cevi\' c in \cite{R1}. Some of its basic
properties were studied in this  article, and  latter in
\cite{R2}, \cite{M} and  \cite{B1}. In the recent past a
particular class of Moore-Penrose invertible objects were studied,
namely, EP Banach space operators and EP Banach algebra elements,
in other words, Moore-Penrose invertible operators or elements of
an algebra such that they commute with their Moore-Penrose
inverse, see \cite{B1, B3, B2R, MD1}. It is worth noticing that EP
objects have been intensively studied in different contexts such
as matrices, Hilbert space bounded and linear maps,
$C^*$-algebras, rings with involution and naturally in the two mentioned above; moreover they consist in a generalization of
the notion normal and hermitian objects, see \cite[Theorems 3.1
and 3.3]{B2R}.\par

\indent In the recent work by Y. Tian and H. Wang  \cite{TW} the
concept of weighted EP matrices (matrices that commute with their
weighted Moore-Penrose inverse) was introduced. What is more, the
second and third authors studied weighted EP $C^*$-algebra
elements in \cite{MD2, MD3}. The main objective of this article is
to introduce and study similar objects in the contexts of Banach
space operators and Banach algebra elements.\par

\indent In fact, in section 2, after having recalled some basic definitions and results, the notion of weighted Moore-Penrose invertible Banach space
operators and Banach algebra elements is introduced and studied. Furthermore, in section 3 weighted EP Banach space operators
and Banach algebra elements are characterized using the concept of group inverse. It is worth noticing that the results presented
in this article apply to arbitrary Banach algebras, operators on Banach spaces and matrices with any norm.

\section{\sfstp The weighted Moore-Penrose inverse}\setcounter{df}{0}
\
\indent From now on  $X$ will denote a Banach space and $L(X)$ the
Banach algebra of all bounded and linear maps defined on and with values in $X$.
If $T\in L(X)$, then $N(T)$ and $R(T)$ will stand for the null space and the
range of $T$ respectively. In addition, $X^*$ will denote the dual space of
$X$ and if $T\in L(X)$, then $T^*\in L(X^*)$ will stand for the adjoint map of $T\in L(X)$.\par

\indent On the other hand, $A$ will denote a complex unital Banach algebra with unit $1$.
In addition, the set of all invertible elements of $A$ will be denoted
by $A^{-1}$. If $a\in A$, then $L_a \colon A\to A$
and $R_a\colon A\to A$ will denote the map defined by
left and right multiplications, respectively:
$$
L_a(x)=ax, \hskip2truecm R_a(x)=xa,
$$
where $x\in A$. Moreover, the following notation will be used:
$N(L_a)= a^{-1}(0)$ and $R(L_a)=aA$.

\indent Recall that an element $a\in A$ is called \it{regular}, \rm
if it has a \it{generalized inverse}, \rm namely if there exists $b\in A$ such that
$$
a=aba.
$$

\indent Furthermore, a generalized inverse $b$ of a regular
element $a\in A$ will be called \it{normalized}, \rm if $b$ is regular
and $a$ is a generalized inverse of $b$, equivalently,
$$
a=aba, \hskip2truecm b=bab.
$$

\indent Next follows the key notion in the definition of the (weighted)
Moore-Penrose inverse in the context of Banach algebras.\par

\begin{df}\label{def1}
Given a unital Banach algebra $A$, an element $a\in A$ is said to be hermitian,
if $\parallel exp(ita)\parallel =1$,  for all $ t\in\Bbb R$.
\end{df}

\indent As regard equivalent definitions and the main properties of hermitian Banach
algebra elements and Banach space operators, see for example \cite{V,L,P,BD, D}.
Recall that if $A$ is a $C^*$-algebra, then
$a\in A$ is hermitian if and only if $a$ is self-adjoint,
see \cite[Proposition 20, Chapter I, Section 12]{BD}.
Given $A$ a unital Banach algebra, the set of all Hermitian
elements of $A$ will be denoted by $H(A)$.
\par

\indent Now the notion of Moore-Penrose invertible Banach algebra element
will be recalled.\par

\begin{df}\label{def2}
 Let $A$ be a unital Banach algebra and $a\in A$. If there exists
$x\in A$ such that $x$ is a normalized generalized inverse of $a$ and
$ax$ and $xa$ are hermitian, then $x$
will be said to be the Moore-Penrose inverse of $a$, and it will be
denoted by $a^{\dag}$.
\end{df}

\indent Recall that according to \cite[Lemma 2.1]{R1},
there is at most one Moore-Penrose inverse. Concerning the Moore-Penrose inverse in Banach algebras,
see \cite{R1,R2,M,B1,B2R,B3,MD1}. In the context of $C^*$-algebras, see \cite{HM1, HM2, Ko2}.
For the original definition of the Moore-Penrose
inverse for matrices, see \cite{Pe}.\par

\indent Next  the weighted Moore-Penrose inverse will be considered. Recall that
given   a $C^*$-algebra $A$ and $a\in A$,
$b\in A$ is said to be the \it weighted Moore-Penrose inverse
 \rm of $a$ with weights $e$ and $f$, if the following identities hold:
$$
aba=a,\hskip.5truecm bab=b,\hskip.5truecm (eab)^*=eab,\hskip.5truecm (fba)^*=fba,
$$
where $e$ and $f$ are positive and invertible elements in $A$, see \cite{KDC,MD2,MD3}.\par
According to \cite{KDC}, the conditions defining the weighted Moore-Penrose inverse
can be rewritten as
 $$
aba=a,\hskip.5truecm bab=b,\hskip.5truecm (ab)^{*e}=ab,\hskip.5truecm (ba)^{*f}=ba,
$$
where $A^{*e}=(A^{*e}, \parallel\cdot\parallel_e)$ (respectively $A^{*f}=(A^{*f}, \parallel\cdot\parallel_f)$) is the $C^*$-algebra
with underlying space $A$, involution $x\to x^{*e}=e^{-1}x^*e$ (respectively $x\to x^{*f}=f^{-1}x^*f$)
and norm $\parallel x\parallel_e =\parallel e^{1/2}xe^{-1/2} \parallel$
(respectively $\parallel x\parallel_f=\parallel f^{1/2}xf^{-1/2} \parallel$),
$x\in A$. Next weighted Moore-Penrose invertible Banach algebra
elements will be introduced. To this end, however, some preparation in needed.\par

\indent Let $A$ be a complex unital Banach algebra and consider
$a\in A$. The element $a$ will be said to be \it positive\rm, if
$V(a)\subset \mathbb R_+$, where $V(a)=\{f(a)\colon f\in A^*,
\parallel f\parallel \le 1, f(1)=1\}$ (\cite[Definition 5, Chapter
V, Section 38]{BD}). Denote by $A_+$ the set of all positive elements
of $A$. Note that necessary and sufficient for $a\in A$ to be
positive is that $a$ is hermitian and $\sigma (a)\subset  \mathbb
R_+$  (\cite[Definition 5, Chapter V, Section 38]{BD}). Recall
that according to \cite[Lemma 7, Chapter V, Section 38]{BD}, if
$c\in A_+$, then there exists $d\in A_+$ such that $d^2=c$.
Moreover, according to \cite[Theorem]{G}, the square root is
unique. In particular, the square root of $c$ will be denoted by
$c^{1/2}$. For the definition and equivalent conditions of
positive $C^*$-algebra elements, see \cite[Definition 3.1 and
Theorem 3.6, Chapter VIII, Section 3]{C}.
\par

\indent Given a complex  unital Banach algebra $A$ and $u\in A^{-1}\cap A_+$,
denote by $A^u=(A^u, \parallel \cdot \parallel_u)$ the complex unital Banach algebra with underlying
space $A$ and norm $\parallel x\parallel_u=\parallel u^{1/2}xu^{-1/2}\parallel$.
When $A$ is a $C^*$-algebra, according to \cite[Proposition 20, Chapter I, Section 12]{BD},
$a$ is self-adjoint in $(A^{*u}, \parallel\cdot\parallel_u)$ if and only if $a$
is hermitian in $(A^u,\parallel\cdot \parallel_u)$, where if $x\in A$, then as before
the involution in $A^{*u}$ is defined as follows:  $x\to x^{*u}=u^{-1}x^*u$.
These facts lead to the following definition.\par

\begin{df}\label{def3}Let $A$ be a complex unital Banach algebra and consider
$e$ and $f$ two positive and invertible elements in $A$. The element $a\in A$
will be said to be weighted Moore-Penrose invertible with weights $e$ and $f$,
if there exists $b\in A$ such that $b$ is a normalized generalized inverse of
$a$ and $ab$ (respectively $ba$) is a hermitian element of $A^e$
(respectively of  $A^f$).
\end{df}
\markright{\hskip4truecm WEIGHTED MOORE-PENROSE INVERSE } 
\indent Clearly, the conditions in
 Definition \ref{def3} extend the notion of weighted Moore-Penrose invertible
$C^*$-algebra element to Banach algebras (\cite{KDC}). Furthermore,
note that if $e=f$, then necessary and sufficient for $a\in A$ to
be weighted Moore-Penrose invertible with weight $e$
is that $a\in A^e$ (respectively $A^{*e}$, when $A$ is a $C^*$-algebra) is Moore-Penrose invertible. In particular,
when both weights coincide, the weighted Moore-Penrose inverse
reduces to the Moore-Penrose inverse, naturally changing 
the structure of the Banach or $C^*$-algebra.
In what follows, some basic properties of weighted Moore-Penrose invertible
Banach algebra elements will be studied.\par

\indent In first place, the uniqueness of the weighted Moore-Penrose inverse will be proved.
Moreover, ideas similar to the ones in \cite[Lemma 2.1]{R1} will be used. However, before considering 
the mentioned property, some preparation is needed.    \par

\begin{lem}\label{lem9}Let $A$ be a unital Banach algebra and consider $u\in A^{-1}\cap A_+$.
Then, $L_u\in L(A)$ is invertible and positive.
\end{lem}

\begin{proof} It is clear that $L_u\in L(A)$ is invertible. In addition,
note that  since $A$ is unital, $\parallel L_u\parallel=\parallel u\parallel$.
In particular, since
$$
\parallel exp(itL_u)\parallel =\parallel L_{exp (itu)}\parallel
=\parallel exp (itu)\parallel=1,
$$
 $L_u\in H(L(A))$. Moreover, since according to \cite[Proposition 4, Chapter I, Section 5 ]{BD}
$\sigma (L_u)=\sigma (u)$, $L_u\in L(A)_+$ (\cite[Definition 5, Chapter V, Section 38]{BD}).
\end{proof}

\begin{pro}\label{prop4} Let $A$ be a complex unital Banach algebra and
consider $e$, $f\in A^{-1}\cap A_+$. Then, if $a\in A$, there is at most
one weighted Moore-Penrose inverse of $a$ with weights $e$ and $f$.
\end{pro}
\begin{proof} Suppose that $b$ and $c$ are two weighted Moore-Penrose inverses
of $a$ with weights $e$ and $f$ and consider $L_{ab}$, $L_{ac}\in
L(A)^{L_e}$. A straightforward calculation proves that
$$
\parallel exp(itL_{ab})\parallel_{L_e} =\parallel exp (itab)\parallel_e = 1,
\hskip.5truecm
\parallel exp(itL_{ac})\parallel_{L_e}=\parallel exp (itac)\parallel_e = 1,
$$
equivalently,  $L_{ab}$ and $L_{ac}$ are two hermitian idempotents
of $L(A)^{L_e}$. Moreover, since $b$ and $c$ are two normalized
generalized inverse of $a$, it is not difficult to prove that
$R(L_{ab})=R(L_a)=R(L_{ac})$. Therefore, according to
\cite[Theorem 2.2]{P}, $L_{ab}=L_{ac}$. In particular,
$ab=ac$.\par \indent A similar argument, using in particular
$R_{ba}$ and $R_{ca}$ instead of $L_{ab}$ and $L_{ac}$
respectively, proves that $ba=ca$.\par \indent Then,
$$
b=bab=cab=cac=c.
$$
\end{proof}
\indent According to Proposition  \ref{prop4}, given a complex unital Banach algebra and $a\in A$, if
 the weighted Moore-Penrose inverse of $a$ exists, then it will be  denoted by
$a_{e,f}^{\dag}$. In the following remarks some elementary facts that will be used in this article
will be presented.\par
 \markboth{  }{ \hskip1.5truecm \rm ENRICO BOASSO, DRAGAN S. DJORDJEVI\'C AND  DIJANA  MOSI\'C }
\begin{rema}\label{rema5} \rm (a) Let $A$ be a complex unital Banach algebra and consider
$e$, $f\in  A^{-1}\cap A_+$. Suppose that $a_{e,f}^{\dag}$ exists, $a\in A$. Then, the following statements
can be easily derived from the conditions characterizing the weighted  Moore-Penrose inverse of $a$.\par
\noindent (i)  $(a_{e,f}^{\dag})_{f,e}^{\dag}=a$.\par
\noindent (ii)  $a_{e,f}^{\dag}A=a_{e,f}^{\dag}aA$, $aa_{e,f}^{\dag}A=aA$.  \par
\noindent (iii) $(aa_{e,f}^{\dag})^{-1}(0) =(a_{e,f}^{\dag})^{-1}(0)$, $(a_{e,f}^{\dag}a)^{-1}(0) =a^{-1}(0)$.\par
\noindent (iv) $A=a_{e,f}^{\dag}A\oplus a^{-1}(0)= aA\oplus (a_{e,f}^{\dag})^{-1}(0)$.\par

\noindent (b) Suppose that $A=L(X)$, $X$ a Banach space. Let $E$, $F\in L(X)$ be two
invertible and positive operators and consider $T\in L(X)$ such that
$T_{E,F}^{\dag}$ exists.   Then, it is not difficult to prove the following facts.\par
\noindent (v)  $R(T_{E,F}^{\dag}T)=R(T_{E,F}^{\dag})$, $R(TT_{E,F}^{\dag})=R(T)$.  \par
\noindent (vi) $N(T_{E,F}^{\dag}T)=N(T)$, $N(TT_{E,F}^{\dag})=N(T_{E,F}^{\dag})$.\par
\noindent (vii) $X=R(T_{E,F}^{\dag})\oplus N(T)=R(T)\oplus N(T_{E,F}^{\dag})$.\par
\end{rema}

\indent Next conditions equivalent to the existence of the weighted Moore-Penrose
inverse will be given. Firstly, the case $A=L(X)$, $X$ a Banach space, will be considered.\par

\begin{thm} \label{thm7}Let $X$ be a Banach space and consider $E$, $F\in L(X)$
two invertible and positive operators.
Then, if $T\in L(X)$, the following statements are equivalent:\par
\noindent  (i) \hskip.18cm $T^{\dag}_{E,F}$ exists;\par
\noindent (ii) there exist two idempotents $P$,$Q\in L(X)$ such that $P\in H(L(X)^E)$,
$R(P)=R(T)$, and $Q\in H(L(X)^F)$, $N(Q)=N(T)$.\par
\noindent Furthermore, if such $P$ and $Q$ exist, then they are unique.\par
 \end{thm}
\begin{proof}
\indent If $T^{\dag}_{E,F}$ exists, then consider
$P= TT^{\dag}_{E,F}$ and $Q=T^{\dag}_{E,F}T$.\par
\indent On the other hand, suppose that statement (ii) holds.
Consider the invertible operator $T'\in L(R(Q), R(T))$,
$$
T'=T\mid_{R(Q)}^{R(T)}\colon R(Q)\to R(T),
$$
and define $S\in L(X)$ as follows:
$$
S\mid_{N(P)}\equiv 0,\hskip1cm S\mid_{R(T)}^{R(Q)} =(T')^{-1}\in
L(R(T),R(Q)).
$$
\indent Since $R(P)=R(T)$, an easy calculation proves that $S$ is a normalized
generalized inverse of $T$. Moreover, since $TS$ and $P$
are idempotents of $L(X)$ such that
$$
R(TS)=R(T)=R(P), \hskip1cm N(TS)=N(S)=N(P),
$$
it is clear that $TS=P$. In particular, $TS\in H(L(X)^E)$.
A similar argument proves that $ST\in H(L(X)^F)$.
Consequently, $S=T^{\dag}_{E,F}$.\par
\markright{\hskip4truecm WEIGHTED MOORE-PENROSE INVERSE } 
\indent Now suppose that $P'$ and $Q'$ are two idempotents
that satisfy statement (ii). Then, $R(P) =R(P')$ and $R(I-Q)=R(I-Q')$.
Then, according to \cite[Hilfssatz 2(a)-(b)]{V} and \cite[Theorem 2.2]{P}, $P=P'$
and $Q=Q'$.
\end{proof}

\indent In the following theorem weighted Moore-Penrose invertible Banach algebra elements will be characterized.    \par

\begin{thm} \label{thm8}Let $A$ be a complex unital Banach algebra
and consider $e$, $f\in  A^{-1}\cap A_+$. If $a$ and $b\in A$ are such that $b$ is a normalized generalized
inverse of $a$, then the following statements are equivalent.\par
\noindent (i) $b=a_{e,f}^{\dag}$;\par
\noindent (ii) $L_b=(L_a)_{L_e,L_f}^{\dag}\in L(A)$.\par
\noindent In particular, if  statements (i)-(ii) hold, then $(L_a)_{L_e,L_f}^{\dag}=L_{a_{e,f}^{\dag}}$.
\end{thm}
\begin{proof}
\indent First of all recall that $L_e$ and $L_f\in L(A)$ are invertible and positive (Lemma \ref{lem9}). \par
\indent Clearly,  $b$ is a normalized generalized inverse of $a$ in $A$ if and only if
$L_b$ is a normalized generalized inverse of $L_a$ in $L(A)$.
In addition,  according to \cite[Theorem]{G}, $(L_e)^{1/2}=L_{e^{1/2}}$ and  $(L_f)^{1/2}=L_{f^{1/2}}$ (Lemma \ref{lem9}).
Furthermore, since
$$
\parallel exp(itL_aL_b)\parallel_{L_e} =\parallel exp (it L_{ab})\parallel_{L_e}=\parallel L_{exp (itab)}\parallel_{L_e}
=\parallel exp (itab)\parallel_e,
$$
and
$$
\parallel exp(itL_bL_a)\parallel_{L_f} =\parallel exp (it L_{ba})\parallel_{L_f}=\parallel L_{exp (itba)}\parallel_{L_f}
=\parallel exp (itba)\parallel_f,
$$
$ab\in H(A^e)$  and $ba\in H(A^f)$ if and only if $L_aL_b\in H (L(A)^{L_e})$
and $L_bL_a\in H(L(A)^{L_f})$. Therefore, statements (i) and (ii) are equivalent.
\end{proof}

\indent In the following proposition the weighted Moore-Penrose inverse
will be described in a particular case. Recall that according to \cite[Thorem 6]{HM1},
the condition of being regular is equivalent to the one of being Moore-Penrose
invertible, for $C^*$-algebra elements. However, according to
\cite[Remark 4]{B1},
in a general Banach algebra these two notions are not in general equivalent.
Compare the next proposition with \cite[Theorem 5]{KDC}.\par

\begin{pro}\label{prop6} Let $A$ be a complex unital Banach algebra and consider
$e$, $f\in A^{-1}\cap A_+$. Then, if $a\in A$ is such that $e^{1/2}af^{-1/2}\in A$ is
Moore-Penrose invertible,
$$
a_{e,f}^{\dag} =f^{-1/2}(e^{1/2}af^{-1/2})^{\dag}e^{1/2}.
$$
\end{pro}
\begin{proof} Let $c=e^{1/2}af^{-1/2}$ and $b=f^{-1/2}c^{\dag}e^{1/2}$.
Then, using in particular that $a=e^{-1/2}cf^{1/2}$, it is not difficult to prove
that $aba=a$, $bab=b$,
$ab= e^{-1/2}cc^{\dag}e^{1/2}$ and $ba=f^{-1/2}c^{\dag}cf^{1/2}$.
However, since $cc^{\dag}$ and $c^{\dag}c$ are hermitian elements of $A$,
$ab\in H(A^e)$ and $ba\in H(A^f)$.
\end{proof}

\indent Next weighted Moore-Penrose inverses in quotient algebras will be considered.
First of all, given a complex unital Banach algebra $A$ and $J$ a proper and closed
two sided ideal of $A$, the quotient map will be denoted by $\Pi\colon A\to A/J$.
In addition, if $a\in A$, then the quotient class of $a$ will be denoted by $\tilde{a}=\Pi (a)$.
\par

\begin{lem}\label{lem11}Let $A$ be a complex unital Banach algebra and consider
$J\subset A$ a proper and closed two sided ideal. Then, if $u\in A^{-1}\cap A_+$,
$\tilde{u}$ is invertible and positive in $A/J$.
\end{lem}
\begin{proof}Clearly, $\tilde{u}\in A/J$ is invertible.  In addition, if $\parallel\cdot\parallel'$
denotes the norm in $A/J$, then
$$
\parallel exp (it\tilde{u})\parallel' =\parallel \Pi (exp (it u))\parallel'\le \parallel exp(itu)\parallel=1.
$$
Consequently, according to \cite[Remark 2]{B1},  $\tilde{u}\in H(A/J)$. Furthermore,
since according to  \cite[Lemma 7, Chapter V, Section 38]{BD} there exists
$v\in A_+$ such that $v^2=u$, $\tilde{v}^2=\tilde{u}$ and
$\sigma (\tilde{u})\subset \mathbb R_+$. Therefore, according to
\cite[Definition 5, Chapter V, Section 38]{BD}, $\tilde{u}\in (A/J)_+$.
\end{proof}

\begin{thm} \label{thm12} Let $A$ be a complex unital Banach algebra and consider
$e$, $f\in A^{-1}\cap A_+$. If $a^{\dag}_{e,f}$ exists, then $\tilde{a}\in A/J$
is weighted Moore-Penrose invertible with weights $\tilde{e}$ and $\tilde{f}$. Furthermore,
$\tilde{a}^{\dag}_{\tilde{e},\tilde{f}}=\Pi(  a^{\dag}_{e,f})$.
\end{thm}
\begin{proof}It is clear that  $\Pi(  a^{\dag}_{e,f})$ is a normalized generalized inverse of
$\tilde{a}$. On the other hand, note that according to \cite[Theorem]{G}, $\Pi(e^{1/2})=\tilde{e}^{1/2}$,
$\Pi(e^{-1/2})=\tilde{e}^{-1/2}$, $\Pi(f^{1/2})=\tilde{f}^{1/2}$ and $\Pi(f^{-1/2})=\tilde{f}^{-1/2}$. As a result,
\begin{align*}
\parallel exp (it \tilde{a}\Pi(  a^{\dag}_{e,f}))\parallel'_{\tilde{e}}&= \parallel exp (it\Pi (a a^{\dag}_{e,f}))\parallel'_{\tilde{e}}
= \parallel \Pi(e^{1/2} exp (ita a^{\dag}_{e,f}) e^{-1/2})\parallel'\\
&\le\parallel e^{1/2} exp (ita a^{\dag}_{e,f}) e^{-1/2}\parallel=\parallel exp (ita a^{\dag}_{e,f}) \parallel_e=1;\\
\parallel exp (it \Pi(  a^{\dag}_{e,f}) \tilde{a})\parallel'_{\tilde{f}}&= \parallel exp (it\Pi ( a^{\dag}_{e,f}a))\parallel'_{\tilde{f}}
= \parallel \Pi(f^{1/2} exp (it a^{\dag}_{e,f}a) f^{-1/2})\parallel'\\
&\le\parallel f^{1/2} exp (it a^{\dag}_{e,f}a) f^{-1/2}\parallel=\parallel exp (it a^{\dag}_{e,f}a) \parallel_f=1.\\
\end{align*}
\indent However, according to \cite[Remark 2]{B1}, the proof is complete.
\end{proof}

\indent The weighted Moore-Penrose inverse in closed invariant subspaces  will be now studied. To this end, some preparation
is needed.  Note that if $X$ is a Banach space and $Y\subseteq X$ is a closed and invariant
subspace for $T\in L(X)$, then $T'=T\mid_Y\in L(Y)$ will stand for the restriction map
of $T$ to $Y$. \par
 \markboth{  }{ \hskip1.5truecm \rm ENRICO BOASSO, DRAGAN S. DJORDJEVI\'C AND  DIJANA  MOSI\'C }
\begin{pro}\label{pro13} Let $X$ be a Banach space and consider $U\in L(X)$
an invertible and positive operator. Let $Y\subseteq X$ be a closed subspace such that
$U(Y)=Y$. Then $U'\in L(Y)$ is invertible and positive. What is more,  $U^{1/2}(Y)\subseteq Y$,
$(U')^{1/2}=(U^{1/2})'\in L(Y)$, $U^{-1/2}(Y)\subseteq Y$ and
$(U')^{-1/2}=(U^{-1/2})'\in L(Y)$.

\end{pro}
\begin{proof} Clearly, $U^{-1}(Y)=Y$ and $U'\in L(Y)$ is invertible. On the other hand, according to
\cite[Lemma 5, Chapter I, Section 10]{BD}, if $S\in A=L(W)$, $W$ a Banach space, then
$V(S)=\cap_{z\in\mathbb C} B[z, \parallel z-S\parallel]$,
where if $z\in \mathbb C$ and $r\in \mathbb R_+$, then $B[z, r]=\{z'\colon \mid z-z'\mid\le r\}$. Since
$$
V(U')=\cap_{z\in\mathbb C} B[z, \parallel z-U'\parallel]\subseteq \cap_{z\in\mathbb C} B[z, \parallel z-U\parallel]
=V(U)\subset  \mathbb R_+,
$$
$U'\in L(Y)$ is positive.\par \indent To prove that
$(U')^{1/2}=(U^{1/2})'\in L(Y)$, consider $K=\sigma (U)\cup\sigma
(U')$. Since $U\in L(X)$ and $U'\in L(Y)$ are invertible and
positive, $K\subset G= \mathbb C\setminus \{x\in\mathbb R\colon
x\le 0\}$. Let $f\colon G\to \mathbb C$ be the principal branch of
the square root function in $G$, i.e., $f(z)=z^{1/2}$ ($z\in G$).
Now well, if $\lambda\in G\setminus K$, then according to
\cite[Lemma 1.28]{D}, $(U-\lambda)^{-1}(Y)\subseteq Y$. As a
result, according to the formula of the Riesz Functional Calculus
for $f\colon G\to \mathbb C$, using in particular an appropriate
system of curves $\Gamma\subseteq G\setminus K$, it is not
difficult to prove that $f(U)(Y)\subseteq Y$. However, since
$f(U)^2=U$ and $\sigma (f(U))=f(\sigma (U))\subseteq \mathbb
R_+\setminus \{0\}$, according to  \cite[Theorem]{G},
$f(U)=U^{1/2}$; in particular, $U^{1/2}(Y)\subseteq Y$. Now well,
since $f(U')=f(U)\mid_Y=(U^{1/2})'$, $\sigma((U^{1/2})')=\sigma
(f(U'))=f(\sigma (U'))\subseteq R_+$. Therefore, since
$((U^{1/2})')^2=U'$, according again to \cite[Theorem]{G},
$(U')^{1/2}=(U^{1/2})'$.\par

\indent Concerning the last two facts, note that according to what has been proved and \cite[Lemma 1.28]{D},
$U^{-1/2}(Y)\subseteq Y$. However, a direct calculation,
using in particular  that $(U')^{1/2}=(U^{1/2})'$, proves that $(U')^{-1/2}=(U^{-1/2})'$.
\end{proof}

\begin{thm}\label{thm14} Let $X$ be a Banach space and consider $E$, $F$, $T\in L(X)$ such that
$E$ and $F$ are invertible and positive and $T^{\dag}_{E,F}$ exists. Suppose in addition that
there is $Y\subseteq X$ a closed and invariant subspace for $T$  and $T^{\dag}_{E,F}$ such that $E(Y)=Y=F(Y)$.
Then, the weighted Moore-Penrose inverse of $T'\in L(Y)$ with respect to the weights
$E'$, $F'\in L(Y)$ exists. What is more, $(T')^{\dag}_{E',F'}=(T^{\dag}_{E,F})'$.
\end{thm}
\begin{proof} It is clear that  $(T^{\dag}_{E,F})'$ is a normalized generalized inverse of $T'$. On
the other hand, according to Proposition \ref{pro13},

\begin{align*}
\parallel exp(it T'(T^{\dag}_{E,F})')\parallel_{E'}&=\parallel (E')^{1/2}exp(it T'(T^{\dag}_{E,F})')(E')^{-1/2}\parallel
\le \parallel E^{1/2}exp(it TT^{\dag}_{E,F})E^{-1/2}\parallel\\
&=\parallel exp(it TT^{\dag}_{E,F})\parallel_E=1,\\
\parallel exp(it (T^{\dag}_{E,F})'T')\parallel_{F'}&=\parallel (F')^{1/2}exp(it (T^{\dag}_{E,F})'T')(F')^{-1/2}\parallel
\le \parallel F^{1/2}exp(it T^{\dag}_{E,F}T)F^{-1/2}\parallel\\
&=\parallel exp(it T^{\dag}_{E,F}T)\parallel_F=1.\\
\end{align*}
\indent Therefore, according to \cite[Remark 2]{B1}, the proof is complete.
\end{proof}

\indent The weighted Moore-Penrose inverse in quotient spaces will be studied now. However, first
some preliminary facts must be recalled. \par

\begin{rema}\label{rem15}\rm Let $X$ be a Banach space and consider
$U\in L(X)$ an invertible and positive operator. Suppose that $Y\subseteq X$
is a closed invariant subspace for $U$ and denote by $\Pi \colon X\to X/Y$ the
quotient map. Then, if $\tilde{U}\in L(X/Y)$ is the quotient operator induced by
$U$, $\tilde{U}$ is invertible and positive. In fact, it is clear that  $\tilde{U}\in L(X/Y)$
is invertible. In addition, according to \cite[Theorem 4.12(ii)]{D},  $\tilde{U}$ is hermitian,
and since $\sigma (\tilde{U})\subseteq \sigma (U)\subset\mathbb R_+$, $\tilde{U}$
is positive (\cite[Definition 5, Chapter V, Section 38]{BD}).\par
\indent Moreover, according to \cite[Theorem]{G}, $(\tilde{U})^{1/2}=\tilde{U^{1/2}}$
and $(\tilde{U})^{-1/2}=\tilde{U^{-1/2}}$.
\end{rema}
\markright{\hskip4truecm WEIGHTED MOORE-PENROSE INVERSE } 
\begin{thm}\label{thm16} Let $X$ be a Banach space and consider $E$, $F\in L(X)$
two invertible and positive operators. Let $T\in L(X)$ such that $T^{\dag}_{E,F}$ exists.
Suppose in addition that $Y\subseteq X$ is a closed and invariant subspace for
$T$, $T^{\dag}_{E,F}$, $E$ and $F$. Then $\tilde{T}\in L(X/Y)$ is weighted
Moore-Penrose invertible with weights $\tilde{E}$ and $\tilde{F}$.
Furthermore, $(\tilde{T})^{\dag}_{\tilde{E}, \tilde{F}}=\tilde{T^{\dag}_{E,F}}$.
\begin{proof} It is clear that $\tilde{T^{\dag}_{E, F}}$
is a normalized generalized inverse of $\tilde{T}$. On the other hand,

\begin{align*}
\parallel exp(it \tilde{T}\tilde{T^{\dag}_{E, F}} )\parallel_{\tilde{E}}&=
\parallel \tilde{E}^{1/2} exp(it \tilde{T}\tilde{T^{\dag}_{E, F}}) \tilde{E}^{-1/2}\parallel
=\parallel \Pi(E^{1/2} exp(it TT^{\dag}_{E, F}) E^{-1/2})\parallel\\
&\le\parallel E^{1/2} exp(it TT^{\dag}_{E, F}) E^{-1/2}\parallel=\parallel exp(it TT^{\dag}_{E, F}) \parallel_E=1,\\
\parallel exp(it\tilde{T^{\dag}_{E, F}}  \tilde{T})\parallel_{\tilde{F}}&=
\parallel \tilde{F}^{1/2} exp(it \tilde{T^{\dag}_{E, F}} \tilde{T}) \tilde{F}^{-1/2}\parallel
=\parallel \Pi(F^{1/2} exp(it T^{\dag}_{E, F}T) F^{-1/2})\parallel\\
&\le\parallel F^{1/2} exp(it T^{\dag}_{E, F}T) F^{-1/2}\parallel=\parallel exp(it T^{\dag}_{E, F}T) \parallel_F=1.\\
\end{align*}
Consequently, according to \cite[Remark 2]{B1}, $(\tilde{T})^{\dag}_{\tilde{E}, \tilde{F}}=\tilde{T^{\dag}_{E,F}}$.
\end{proof}
\end{thm}

\section{\sfstp Weighted EP elements}\setcounter{df}{0}
\
  \markboth{  }{ \hskip1.5truecm \rm ENRICO BOASSO, DRAGAN S. DJORDJEVI\'C AND  DIJANA  MOSI\'C }
In this section, weighted EP Banach space operators and Banach
algebra elements will be considered. In particular,  these elements will be characterized extending results for matrices
(\cite{TW}) and elements of $C^*$-algebras (\cite{MD2}). In first
place, the main notion of this section will be introduced.\par

\begin{df}\label{def17}
Given a unital Banach algebra $A$ and $e$, $f\in A$ two invertible and
positive elements,  $a\in A$ is said to be weighted EP with weights $e$ and $f$,
if $a^{\dag}_{e,f}$ exists and commutes with $a$.
\end{df}

\indent Under the same conditions of Definition \ref{def17} and as it was
pointed out in the paragraph that follows Definition \ref{def3},  note that if $e=f$, then
necessary and sufficient for $a\in A$ to be weighted EP with weight $e$ is that $a\in A^e$
is EP; EP Banach algebra elements were studied in \cite{B1, B3, B2R, MD1}.
In the context of $C^*$-algebras, see \cite{Ko2}.\par

\indent  To study the objects introduced in Definition \ref{def17},  the notion of group inverse need to be
recalled. In fact, weighted EP Banach algebra elements consists in a particular class of group invertible
elements. \par

\begin{df}\label{df18}Given a unital Banach algebra $A$ and $a\in A$, an element $b\in A$ will be said
to be the group inverse of $a$, if the following set of equations is satisfied:\par
$$
a=aba, \hskip.8cm b=bab, \hskip.8cm ab=ba.
$$
\end{df}
\indent Under the conditions of Definition \ref{df18}, note that according to \cite[Theorem 9]{HM1}, if the group inverse of $a\in A$ exists, then
it is unique. In this case, the group inverse of $a\in A$ will be denoted by $a^{\sharp}$.  In the following remark some of the most relevant properties of the group inverse will be
recalled.\par

\begin{rema}\label{rema19}\rm  (i)  Let $A$ be a unital Banach algebra and consider $a\in A$.
Suppose that $b\in A$ is a normalized generalized inverse of $a$. Then, necessary and sufficient for $b$ to be the group inverse of $a$ is that
$L_b\in L(A)$ (respectively $R_b\in L(A)$) is the group inverse of $L_a\in L(A)$ (respectively $R_a\in L(A)$). In fact, since $L_b$ is a normalized generalized inverse of $L_a$,
according to Definition \ref{df18} the statement under consideration is equivalent to prove that $a$ and $b$ commute
if and only if $L_a$ and $L_b$ commute, which is clear. A similar argument proves the statement for the right multiplication operator on $L(A)$.
In addition, note that in this case, according to  \cite[Theorem 9]{HM1}, $(L_a)^{\sharp} =L_{a^\sharp}$ (respectively $(R_a)^{\sharp} =R_{a^\sharp}$).\par

\noindent (ii) Suppose that $e$, $f\in A$ are invertible and positive ($A$ as in (i)).
Note that if $a\in A$ is group invertible,
then necessary and sufficient for $a$ to be weighted EP with weights $e$ and $f$ is that $a a^{\sharp}\in A^e$
and $a^{\sharp}a\in A^f$ are hermitian. In fact, if $a\in A$ is weighted EP with weights $e$ and $f$, then according to
\cite[Theorem 9]{HM1}, $a^{\sharp}$ exists, actually $a^{\sharp}=  a^{\dag}_{e,f}$.
In particular, $a a^{\sharp}\in A^e$
and $a^{\sharp}a\in A^f$ are hermitian.
On the other hand, if $a a^{\sharp}\in A^e$
and $a^{\sharp}a\in A^f$ are hermitian, then according to Definition \ref{def3}, $a^{\dag}_{e,f}$ exists.
What is more, according to  Proposition \ref{prop4},
$a^{\sharp}=a^{\dag}_{e,f}$. Since
$a a^{\dag}_{e,f}=aa^{\sharp}= a^{\sharp}a=a^{\dag}_{e,f}a$,
$a$ is weighted EP with weights $e$ and $f$. Consequently,
weighted EP elements are group invertible elements for which the weighted Moore-Penrose inverse exists and coincides
with the group inverse. \par

\noindent (iii) Let $A=L(X)$, $X$ a Banach space, and consider $T\in L(X)$. Then, according to \cite[Lemma 1]{Ki},
necessary and sufficient for $T^{\sharp}$ to exists is that $X=N(T)\oplus R(T)$.

\noindent (iv) Under the conditions of (iii), note that if $T\in L(X)$ is group invertible,
then, as in Remark \ref{rema5}(b),
 it is not difficult to prove that $N(T)=N(T^{\sharp}T)=N(TT^{\sharp})=N(T^{\sharp})$ and $R(T)=R(TT^{\sharp})=R(T^{\sharp}T)=R(T^{\sharp})$.\par

\noindent (v) Observe that according to items (iii) and (iv), $R(T^k)=R((T^{\sharp})^k)=R(T)$ and $N(T^k)=N((T^{\sharp})^k)=N(T)$. \end{rema}

\indent In the following theorem the first characterization of weighted EP elements will be given.\par

\begin{thm}(a)  Let $X$ be a Banach space and consider $E$, $F\in
L(X)$ two invertible and positive operators. Then, if $T\in L(X)$,
the following statements are equivalent.\par
 \noindent  (i)
\hskip.18cm $T$ is weighted  EP with weights  $E$ and $F$;\par
\noindent (ii) there exists an idempotent $P\in L(X)$ such that $P\in
H(L(X)^E)\cap H(L(X)^F)$, $R(P)=R(T)$ and $N(P)=N(T)$.\par
\noindent (b)  Let $A$ be a unital Banach algebra and consider $e$, $f\in
A$ two invertible and positive elements. Then, if $a\in A$ is such that $a^{\dag}_{e,f}$ exists,
the following statements are equivalent.\par
 \noindent  (iii) $a$ is weighted  EP with weights  $e$ and $f$;\par
\noindent (iv) there exists an idempotent $P\in L(A)$ such that $P\in
H(L(A)^{L_e})\cap H(L(A)^{L_f})$, $R(P)=aA$ and $N(P)=a^{-1}(0)$,\par
\noindent (v) $L_a\in L(A)$ is weighted
EP with weights $L_e$ and $L_f$.\par
\noindent Furthermore, if in (a) or (b) such $P$ exists, then it is unique.\par
 
\end{thm}
\begin{proof} (a) If $T$ is weighted  EP with weights  $E$ and $F$,
then $P=TT^{\dag}_{E,F}=T^{\dag}_{E,F}T$ satisfies the required
property (Remark \ref{rema5}(b)).\par

\indent On the other hand, if statement (ii) holds, then according to Theorem \ref{thm7}, $T^{\dag}_{E,F}$ exists.
Further, observe that $R(TT^{\dag}_{E,F})=R(T)=R(P)$ and
$R(I-T^{\dag}_{E,F}T)=N(T^{\dag}_{E,F}T)=N(T)=N(P)=R(I-P)$. Therefore, according to
\cite[Hilfssatz 2(a)-(b)]{V} and \cite[Theorem 2.2]{P},
$T^{\dag}_{E,F} T=P=TT^{\dag}_{E,F}$.\par

\noindent (b) According to what has been proved, statements (iv) and (v) are equivalent.
 In addition, since according to Theorem \ref{thm8}, $(L_a)_{L_e,L_f}^{\dag}=L_{a_{e,f}^{\dag}}$,
statements (iii) and (v) are equivalent.\par  

\indent The last statement is a consequence of Theorem \ref{thm7}.
\end{proof}

\indent Next some basic characterizations of weighted EP Banach
space operators will be given.\par
\markright{\hskip4truecm WEIGHTED MOORE-PENROSE INVERSE } 
\begin{thm}\label{thm20}
Let $X$ be a Banach space and consider $E$, $F\in L(X)$ two invertible
and positive operators. Suppose in addition that $T\in L(X)$ is such that $T^\dag_{E,F}$ and
$T^{\sharp}$ exist. Then, the following statements are equivalent.
\begin{itemize}

\item[\rm (i)] $T$ is weighted EP with weights  $E$ and $F$;

\item[\rm(ii)] $R(T^\dag_{E,F})=R(T)$ and $N(T^\dag_{E,F})=N(T)$;

\item[\rm(iii)] $R(T^\dag_{E,F})\subset R(T)$ and $N(T)\subset
N(T^\dag_{E,F})$;

\item[\rm(iv)] $R(T)\subset R(T^\dag_{E,F})$ and $N(T)\subset
N(T^\dag_{E,F})$;

\item[\rm(v)] $R(T^\dag_{E,F})\subset R(T)$ and
$N(T^\dag_{E,F})\subset N(T)$;

\item[\rm(vi)] $R(T)\subset R(T^\dag_{E,F})$ and
$N(T^\dag_{E,F})\subset N(T)$;

\item[\rm (vii)] $T^\dag_{E,F}=T^{\sharp}$;

\item[\rm (viii)] $T^\dag_{E,F}=T(T^\dag_{E,F})^2=(T^\dag_{E,F})^2T$;

\item[\rm (ix)] $T=T^\dag_{E,F}T^2=T^2T^\dag_{E,F}$;

\item[\rm (x)] $T^\dag_{E,F}$ is weighted EP with weights  {\rm$F$} and {\rm $E$}, moreover
$(T^\dag_{E,F})^{\dag}_{F,E}=T$;

\item[\rm (xi)]  $T$ is weighted  EP with weights $F$ and $E$;

\item[\rm (xii)]  $T$ is both weighted EP with weights  $E$ and $E$
and with weights  $F$ and $F$;

\item[\rm (xiii)] $T^k$ is weighted EP with weights $E$ and $F$,
for some integer $k\geq 1$;

\item[\rm (xiv)] $T^{\sharp}$ is weighted EP with weights $E$ and
$F$;

\item[\rm (xv)] $TT^{\sharp}=TT^\dagger_{E,E}=TT^\dagger_{F,F}$
{\rm (}or $TT^{\sharp}=T^\dagger_{E,E}T=T^\dagger_{F,F}T${\rm )};

\item[\rm (xvi)] $TT^{\sharp}=TT^\dagger_{E,F}=TT^\dagger_{F,E}$
{\rm (}or $TT^{\sharp}=T^\dagger_{F,E}T=T^\dagger_{E,F}T${\rm )}.

\end{itemize}
\end{thm}
 \markboth{  }{ \hskip1.5truecm \rm ENRICO BOASSO, DRAGAN S. DJORDJEVI\'C AND  DIJANA  MOSI\'C }
\begin{proof} If  statement (i) holds, then according to Remark \ref{rema5}(b),
$R(T^\dag_{E,F})=R(T^\dag_{E,F}T)=R(TT^\dag_{E,F})=R(T)$
and  $N(T^\dag_{E,F})=  N(TT^\dag_{E,F})= N(T^\dag_{E,F}T)= N(T)$.
Consequently, statement (ii) holds, which in turn clearly implies
statements (iii)-(vi).\par

\indent On the other hand, if one of the statements (iii)-(vi) holds,
say statement (iii) for example, then, using the decompositions
of $X$
$$
X=R(T_{E,F}^{\dag})\oplus N(T)=R(T)\oplus N(T_{E,F}^{\dag})=R(T)\oplus N(T)
$$

\noindent (Remark \ref{rema5}(b) and Remark \ref{rema19}(iii)),
 it is not difficult to prove that $R(T^\dag_{E,F})=R(T)$ and $N(T^\dag_{E,F})=N(T)$,
equivalently statement (ii) holds. However, in this case $R(T^\dag_{E,F}T)=R(TT^\dag_{E,F})$ and $N(T^\dag_{E,F}T)=N(TT^\dag_{E,F})$.
Since $T^\dag_{E,F}T$ and $TT^\dag_{E,F}\in L(X)$ are idempotents,  $T^\dag_{E,F}T=TT^\dag_{E,F}$, i.e., $T$ is
weighted EP with weights $E$ and $F$.\par

\indent An argument similar to the one considered in Remark  \ref{rema19}(ii) proves the equivalence between statements (i) and (vii).\par

\indent Statement (i) implies statements (viii)-(ix). On the other hand, statement (viii) implies statement (iii)
and statement (ix) implies statement (vi).

\indent The equivalence between statements (i) and (x) is a consequence of Definition \ref{def3}.

\indent Since according to Remark \ref{rema19}(ii), $T$ is weighted EP with
weights $F$ and $E$ if and only if $TT^{\sharp}\in H(L(X)^F)\cap
H(L(X)^E)$, statements (i) and (xi) are equivalent. Similarly,
it is possible to prove that statement (i) is equivalent to statements (xii)-(xiv)
using that $T^k(T^k)^{\sharp}=T^k(T^{\sharp})^k=TT^{\sharp}$ (for
some integer $k\geq 1$) and $(T^{\sharp})^{\sharp}=T^{\sharp}$.\par

\indent If statement (xii) holds, then
$T^{\sharp}=T^\dagger_{E,E}=T^\dagger_{F,F}$ and statement (xv) is
satisfied, which in turn implies statement (i) (Remark \ref{rema19}(ii)). In a similar way it is possible to prove that statement (xvi)
implies statement (i) and statement (xi) implies statement (xvi).
\end{proof}

\indent In the following theorem more characterizations concerning weighted EP operators will be proved.\par

\begin{thm}\label{thm21}
Let $X$ be a Banach space and consider $E,F\in L(X)$ two
invertible and positive operators. Suppose in addition that $T\in
L(X)$ is such that $T^\dag_{E,F}$ and $T^{\sharp}$ exist. Then, necessary and sufficient
for $T$ to be weighted EP with weights $E$ and $F$ is that one of the following statements holds.
\markright{\hskip4truecm WEIGHTED MOORE-PENROSE INVERSE } 
\begin{align*}
&{\rm (i)}& &TT^{\sharp}T^{\dag}_{E,F}=T^{\dag}_{E,F}T^{\sharp}T;&\\
 &{\rm (ii)}& & TT^{\dag}_{E,F}T^{\dag}_{E,F}T=T^{\dag}_{E,F}TTT^{\dag}_{E,F}; &\\
 &{\rm (iii)}&& T^{\dag}_{E,F}T^{\sharp}T+TT^{\sharp}T^{\dag}_{E,F}=2T^{\dag}_{E,F};&\\
 &{\rm (iv)}&& (T^{\dag}_{E,F})^2T^{\sharp}=T^{\dag}_{E,F}T^{\sharp}T^{\dag}_{E,F}=T^{\sharp}(T^{\dag}_{E,F})^2;&\\
 &{\rm (v)}&&T(T^{\dag}_{E,F})^2=T^{\sharp}=(T^{\dag}_{E,F})^2T;&\\
 &{\rm (vi)}& &T^k=T^{\dag}_{E,F}TT^k=T^kTT^{\dag}_{E,F};&\\
&{\rm (vii)}& &(T^{\dag}_{E,F})^k=(T^{\sharp})^k;&\\
    &{\rm (viii)}& &(T^{\sharp})^k T^{\dag}_{E,F}=T^{\dag}_{E,F}(T^{\sharp})^k ;&\\
 &{\rm (ix)}& &T^k T^{\dag}_{E,F}=T^{\dag}_{E,F}T^k ;& \\
&{\rm (x)}& &T^{2k-1}= T^{\dag}_{E,F} T^{2k+1}T^{\dag}_{E,F};&\\
&{\rm (xi})& &(T^{\sharp})^kT^{\dag}_{E,F}T=(T^{\dag}_{E,F})^k;&.\\
 &{\rm (xii)}& &T(T^{\dag}_{E,F})^{k+1}=(T^{\sharp})^k=(T^{\dag}_{E,F})^{k+1}T;&\\
&{\rm (xiii)}&    &T^kTT^{\dag}_{E,F}+T^{\dag}_{E,F}TT^k=2T^k;& \\
&{\rm (xiv})&    &(T^{\sharp})^{k+l-1}=(T^{\dag}_{E,F})^l (T^{\sharp})^{k-1}= (T^{\sharp})^{k-1}(T^{\dag}_{E,F})^l;&\\
&{\rm (xv})&    &TT^{\dag}_{E,F}(T^{\sharp})^{k+l-1}=(T^{\dag}_{E,F})^k(T^{\sharp})^lT=(T^{\sharp})^lT(T^{\dag}_{E,F})^k;&\\
&{\rm (xvi})&  &TT^\dagger_{E,F}(T^k+\lambda
T^\dagger_{E,F})=(T^k+\lambda T^\dagger_{E,F})TT^\dagger_{E,F};&\\
&{\rm (xvii})&  &T^\dagger_{E,F}T(T^k+\lambda
T^\dagger_{E,F})=(T^k+\lambda T^\dagger_{E,F})T^\dagger_{E,F}T;&\\
&{\rm (xviii})&  &R(T+\lambda T^\dagger_{E,E})=R(T+\lambda
T^\dagger_{F,F})=R(\lambda T+T^3)\  and\  &\\
& &  &N(T+\lambda
T^\dagger_{E,E})=N(T+\lambda T^\dagger_{F,F})=N(\lambda T+T^3);&\\
&{\rm (xix})&  &R(T+\lambda T^\dagger_{E,F})=R(\lambda T+T^3)\
and\
N(T+\lambda T^\dagger_{E,F})=N(\lambda T+T^3),&\\
\end{align*}
where $k$, $l\in \mathbb N$ and $\lambda\in \mathbb
C\backslash\{0\}$.
\end{thm}
\begin{proof} It is not difficult to prove that the condition of being weighted EP with weights $E$ and $F$
implies statements (i)--(xvii). On the other hand, if
statement (i) holds, then
$$
R(TT^{\sharp}T^{\dag}_{E,F})=R(T^{\sharp}TT^{\dag}_{E,F})=T^{\sharp}R(TT^{\dag}_{E,F})=R(T).
$$
As a result, $R(T)\subseteq R(T^{\dag}_{E,F})$. In addition, if
$x\in N(T^{\dag}_{E,F})$, then $T^{\sharp}T(x)\in R(T)\cap
N(T^{\dag}_{E,F})=\{0\}$, Thus, $x\in N(T^{\sharp}T)=N(T)$ and
$N(T^{\dag}_{E,F})\subseteq N(T)$. Consequently, according to
Theorem \ref{thm20}(vi), the condition of being weighted EP is
satisfied.\par
 \markboth{  }{ \hskip1.5truecm \rm ENRICO BOASSO, DRAGAN S. DJORDJEVI\'C AND  DIJANA  MOSI\'C }
\indent Suppose that statement (ii) holds. If $x\in N(T)$, then
$TT^{\dag}_{E,F}(x)\in R(T)\cap N(T^{\dag}_{E,F}T)= R(T)\cap N(T)
=\{0\}$. Thus, $x\in N(TT^{\dag}_{E,F})=N(T^{\dag}_{E,F})$ and
$N(T)\subseteq N(T^{\dag}_{E,F})$. On the other hand, since
$$R(T^{\dag}_{E,F}TTT^{\dag}_{E,F})=T^{\dag}_{E,F}T( R(TT^{\dag}_{E,F}))=T^{\dag}_{E,F}T(R(T))
=R(T^{\dag}_{E,F}T)=R(T^{\dag}_{E,F}),$$
$R(T^{\dag}_{E,F})\subseteq R(T)$. Therefore, according to Theorem
\ref{thm20}(iii), $T$ is weighted EP with weights $E$ and $F$.\par

\indent If statement (iii) holds, then
$R(T)=R(TT^{\dag}_{E,F})=R(T^{\sharp}TT^{\dag}_{E,F})\subseteq
R(T^{\dag}_{E,F})$. In addition, if $x\in N(T^{\dag}_{E,F})$, then
$T^{\sharp}T(x)\in R(T)\cap N(T^{\dag}_{E,F})=\{0\}$. Thus, $x\in
N(T^{\sharp}T)=N(T)$ and $N(T^{\dag}_{E,F})\subseteq N(T)$. In
particular, according to Theorem \ref{thm20}(vi), the condition of
being weighted EP is satisfied.\par

\indent Suppose that statement (iv) holds. According to the equality $(T^{\dag}_{E,F})^2T^{\sharp}=T^{\dag}_{E,F}T^{\sharp}T^{\dag}_{E,F}$,
\begin{align*}
(T^{\dag}_{E,F})^2T^{\sharp}&=((T^{\dag}_{E,F})^2T^{\sharp})TT^{\sharp}=T^{\dag}_{E,F}T^{\sharp}T^{\dag}_{E,F}TT^{\sharp}\\
               &=T^{\dag}_{E,F}(T^{\sharp})^2TT^{\dag}_{E,F}TT^{\sharp}=T^{\dag}_{E,F}(T^{\sharp})^2.\\
\end{align*}
As a result,
\begin{align*}
TT^{\dag}_{E,F}&=T^3(T^{\sharp})^2T^{\dag}_{E,F}=T^3T^{\dag}_{E,F}T(T^{\sharp})^2T^{\dag}_{E,F}=T^3(T^{\dag}_{E,F}T^{\sharp}T^{\dag}_{E,F})\\
   &=T^3((T^{\dag}_{E,F})^2T^{\sharp})=T^3T^{\dag}_{E,F}(T^{\sharp})^2=T^3T^{\dag}_{E,F}T(T^{\sharp})^3=TT^{\sharp}.\\
\end{align*}

\indent A similar argument, using in particular the equality $T^{\dag}_{E,F}T^{\sharp}T^{\dag}_{E,F}=T^{\sharp}(T^{\dag}_{E,F})^2$,
proves that $T^{\dag}_{E,F}T=TT^{\sharp}$. Therefore, $T$ is weighted EP with
weights $E$ and $F$.\par

\indent Next consider statement (v). The condition $T(T^{\dag}_{E,F})^2=T^{\sharp}=(T^{\dag}_{E,F})^2T$ implies that
$$
TT^{\dag}_{E,F}=TT^{\sharp}TT^{\dag}_{E,F}=TT(T^{\dag}_{E,F})^2TT^{\dag}_{E,F}=T(T(T^{\dag}_{E,F})^2)=TT^{\sharp}
$$
and
$$
T^{\dag}_{E,F}T=T^{\dag}_{E,F}TT^{\sharp}T=T^{\dag}_{E,F}T(T^{\dag}_{E,F})^2TT=((T^{\dag}_{E,F})^2T)T=T^{\sharp}T.
$$
Consequently, $T$ is weighted EP with weights $E$ and $F$.\par

\indent If statement (vi) holds, then $R(T)=R(T^k)\subseteq
R(T^{\dag}_{E,F})$ and $N(T^{\dag}_{E,F})\subseteq N(T^k)=N(T)$.
Therefore, according to Theorem \ref{thm20}(vi), $T$ is weighted
EP with weights $E$ and $F$. A similar argument, using in
particular that the ranges and null spaces of $T$ and $T^{\sharp}$
coincide, proves that statement (vii) is equivalent to the
condition of being weighted EP.\par

\indent Suppose that statement (viii) holds. Then,
$$R(T^{\dag}_{E,F})=R(T^{\dag}_{E,F}T)=T^{\dag}_{E,F}(R(T))
=T^{\dag}_{E,F}(R((T^{\sharp})^k))=R(T^{\dag}_{E,F}(T^{\sharp})^k)\subseteq R((T^{\sharp})^k) =R(T).$$
In addition, if $x\in N(T)=N((T^{\sharp})^k)$, then $T^{\dag}_{E,F}(x)\in R(T^{\dag}_{E,F})\cap N((T^{\sharp})^k)=
 R(T^{\dag}_{E,F})\cap N(T)=\{0\}$. As a result, $N(T)\subseteq  N(T^{\dag}_{E,F})$. Consequently,
according to Theorem \ref{thm20}(iii), $T$ is weighted EP with
weights $E$ and $F$. A similar argument can be applied to
statement (ix).\par

\indent Statement (x) implies that $N(T^{\dag}_{E,F})\subseteq N(T^{2k-1})=N(T)$ and $R(T)=R(T^{2k-1})\subseteq
R(T^{\dag}_{E,F})$. As a result, according to Theorem \ref{thm20}(vi), the condition of being weighted EP is satisfied.\par

\indent Suppose that statement (xi) holds. Now well, since
$X=R(T^{\dag}_{E,F}T)\oplus N((T^{\sharp})^k)$,
$R(T)=R((T^{\sharp})^k)\subseteq R(T^{\dag}_{E,F})$. In addition,
if $x\in N(T^{\dag}_{E,F})$, then $T^{\dag}_{E,F}T(x)\in
R(T^{\dag}_{E,F})\cap N((T^{\sharp})^k)= R(T^{\dag}_{E,F})\cap
N(T)=\{0\}$. Hence, $x\in N( T^{\dag}_{E,F}T)=N(T)$ and
$N(T^{\dag}_{E,F})\subseteq N(T)$. Consequently,  according to
Theorem \ref{thm20}(vi), $T$ is weighted EP with weights $E$ and
$F$.

\indent Suppose that statement (xii) holds.
Then, it is not difficult to prove that $R(T)\subseteq R(T^{\dag}_{E,F})$ and $N(T^{\dag}_{E,F})\subseteq N(T)$.
In particular, according to Theorem \ref{thm20}(vi), $T$ is weighted EP with weights $E$ and $F$.\par
\markright{\hskip4truecm WEIGHTED MOORE-PENROSE INVERSE } 
\indent If statement (xiii) holds, then $R(T^{\dag}_{E,F})=R(T^{\dag}_{E,F}T)=R(T^{\dag}_{E,F}TT^k)\subseteq R(T^k)=R(T)$.
In addition, if $x\in N(T)=N(T^k)$, then $T^{\dag}_{E,F}(x)\in N(T^{k+1})=N(T)$, which implies that
$x\in N(T^{\dag}_{E,F})$. Thus,  $N(T)\subseteq N(T^{\dag}_{E,F})$. However, according to
 Theorem \ref{thm20}(iii), the condition of being weighted EP holds.\par

\indent Statement (xiv) implies that  $R(T)\subseteq R(T^{\dag}_{E,F})$ and $N(T^{\dag}_{E,F}))\subseteq N(T)$.
Therefore, according to  Theorem \ref{thm20}(vi), $T$ is weighted EP with weights $E$ and $F$.\par

\indent Suppose that statement (xv) holds. Then,
$$R(T)=R(TT^{\dag}_{E,F}T)=TT^{\dag}_{E,F}R(T)
=TT^{\dag}_{E,F}R((T^{\sharp})^{k+l-1})=R(TT^{\dag}_{E,F}(T^{\sharp})^{k+l-1})
\subseteq R(T^{\dag}_{E,F}).$$ In addition, if $x\in
N(T^{\dag}_{E,F})$, then $(T^{\sharp})^{k+l-1}(x)\in
N(TT^{\dag}_{E,F})\cap R(T)= N(T^{\dag}_{E,F})\cap R(T)=\{0\}$.
Then $N(T^{\dag}_{E,F})\subseteq N(T)$. Therefore,  according to
Theorem \ref{thm20}(vi), the condition of being weighted EP is
satisfied.

\indent Statement (xvi) can be rewritten as
\begin{equation}T^k+\lambda
TT^\dagger_{E,F}T^\dagger_{E,F}=T^kTT^\dagger_{E,F}+\lambda
T^\dagger_{E,F}.\label{jedn4}\end{equation} 

Then, multiplying
(\ref{jedn4}) from the left side by $TT^\dagger_{E,F}$, 
$$T^k+\lambda
TT^\dagger_{E,F}T^\dagger_{E,F}=T^kTT^\dagger_{E,F}+\lambda
TT^\dagger_{E,F}T^\dagger_{E,F}.$$ Hence,
$T^k=T^kTT^\dagger_{E,F}$, which implies that
$N(T^\dagger_{E,F})\subseteq N(T^k)=N(T)$. Similarly, multiplying
(\ref{jedn4}) from the right side by $TT^\dagger_{E,F}$,
$$T^kTT^\dagger_{E,F}+\lambda
TT^\dagger_{E,F}T^\dagger_{E,F}=T^kTT^\dagger_{E,F}+\lambda
T^\dagger_{E,F}.$$ As a result,
$TT^\dagger_{E,F}T^\dagger_{E,F}=T^\dagger_{E,F}$,  which implies that 
$R(T^\dagger_{E,F})\subseteq R(T)$. Consequently, according to Theorem \ref{thm20}(v), $T$ is weighted EP with weights $E$ and $F$.\par

\indent A similar argument proves that statement (xvii) implies statement (i).\par

\indent Suppose that statements (xviii) holds. The equality $R(T+\lambda
T^\dagger_{E,E})=R(\lambda T+T^3)$ implies that for $x\in X$ there
exists $y\in X$ such that $(T+\lambda T^\dagger_{E,E})x=(\lambda
T+T^3)y$. Then
\begin{eqnarray*}
(T+\lambda T
T^\dagger_{E,E}T^\dagger_{E,E})x&=&TT^\dagger_{E,E}(T+\lambda
T^\dagger_{E,E})x=TT^\dagger_{E,E}(\lambda T+T^3)y\\
&=&(\lambda T+T^3)y=(T+\lambda T^\dagger_{E,E})x,
\end{eqnarray*}
which implies that $TT^\dagger_{E,E}T^\dagger_{E,E}x=T^\dagger_{E,E}x$
($x\in X$). Thus, $R(T^\dagger_{E,E})\subseteq R(T)$. Let
$x\in N(T)$. Now well, $(\lambda T+T^3)x=0$ and, since
$N(T+\lambda T^\dagger_{E,E})=N(\lambda T+T^3)$,
$(T+\lambda T^\dagger_{E,E})x=0$, i.e., $T^\dagger_{E,E}x=0$.
Therefore, $N(T)\subseteq N(T^\dagger_{E,E})$, and according to  Theorem \ref{thm20}(iii), $T$ is
weighted EP with weights $E$ and $E$.
Similarly, using that $R(T+\lambda T^\dagger_{F,F})=R(\lambda T+T^3)$
and $N(T+\lambda T^\dagger_{F,F})=N(\lambda T+T^3)$, it is possible to prove that 
that $T$ is weighted EP with weights $F$ and $F$. According then to Theorem \ref{thm20}(xii), $T$ is weighted EP with weights $E$ and $F$.

\indent On the other hand, if $T$ is weighted EP with weights $E$ and $F$,
then $T^{\sharp}=T^\dagger_{E,E}=T^\dagger_{F,F}$ (Theorem
\ref{thm20}(xii)). Since
$$T+\lambda T^{\sharp}=(T^3+\lambda T)(T^{\sharp})^2=(T^{\sharp})^2(T^3+\lambda T)$$
and
$$T^3+\lambda T=(T+\lambda T^{\sharp})T^2=T^2(T+\lambda T^{\sharp}),$$ statement (xviii) holds. A similar
argument proves that statement (xix) is equivalent to the condition of being weighted EP.
\end{proof}

\indent Next weighted EP Banach algebra elements will be characterized.\par

  \markboth{  }{ \hskip1.5truecm \rm ENRICO BOASSO, DRAGAN S. DJORDJEVI\'C AND  DIJANA  MOSI\'C }

\begin{thm}\label{thm22}
Let $A$ be a unital Banach algebra and consider $e$, $f\in A$ two invertible
and positive elements. Suppose in addition that $a\in A$ is such that $a^\dag_{e,f}$ and
$a^{\sharp}$ exist. Then, the following statements are equivalent.
\begin{itemize}

\item[\rm (i)] $a$ is weighted EP with weights  {\rm$e$} and {\rm $f$};

\item[\rm(ii)] $a^{\dag}_{e,f}A=aA$ and $(a^{\dag}_{e,f})^{-1}(0)=a^{-1}(0)$;

\item[\rm(iii)] $a^\dag_{e,f}A\subset aA$ and $a^{-1}(0)\subset
(a^\dag_{e,f})^{-1}(0)$;

\item[\rm(iv)] $aA\subset a^\dag_{e,f}A$ and $a^{-1}(0)\subset
(a^\dag_{e,f})^{-1}(0)$;

\item[\rm(v)] $a^\dag_{e,f}A\subseteq aA$ and
$(a^\dag_{e,f})^{-1}(0)\subset a^{-1}(0)$;

\item[\rm(vi)] $aA\subset a^\dag_{e,f}A$ and
$(a^\dag_{e,f})^{-1}(0)\subset a^{-1}(0)$;

\item[\rm (vii)] $a^\dag_{e,f}=a(a^\dag_{e,f})^2=(a^\dag_{e,f})^2a$;

\item[\rm (viii)] $a^\dag_{e,f}$ is weighted EP with weights  {\rm$f$} and {\rm $e$};

\item[\rm (ix)] $a^{\dag}_{e,f}a^{\sharp}a+aa^{\sharp}a^{\dag}_{e,f}=2a^{\dag}_{e,f}$;

\item[\rm (x)] $(a^{\dag}_{e,f})^2a^{\sharp}=a^{\dag}_{e,f}a^{\sharp}a^{\dag}_{e,f}=a^{\sharp}(a^{\dag}_{e,f})^2$;

\item[\rm (xi)] $a(a^{\dag}_{e,f})^2=a^{\sharp}=(a^{\dag}_{e,f})^2a$;

\item[\rm (xii)] $aa^{\dag}_{e,f}a^{\dag}_{e,f}a=a^{\dag}_{e,f}aaa^{\dag}_{e,f}$;

\item[\rm(xiii)] $aa^{\sharp}a^{\dag}_{e,f}=a^{\dag}_{e,f}a^{\sharp}a$;

\item[\rm(xiv)] $a\in a^{\dag}_{e,f}A^{-1}\cap A^{-1}a^{\dag}_{e,f}$;

\item[\rm(xv)] there exist $u$, $v\in A$ such that
$a=a^{\dag}_{e,f}u=va^{\dag}_{e,f}$ and $uA=A$ and
$v^{-1}(0)=\{0\}$;

\item[\rm(xvi)] $a\in a^{\dag}_{e,f}A\cap Aa^{\dag}_{e,f}$;

\item[\rm(xvii)] if $b\in A$ is such that $ab=ba$, then $a^{\dag}_{e,f}b=ba^{\dag}_{e,f}$;

\item[\rm(xviii)] there exists a holomorphic function $f\colon U\to \mathbb C$, $\sigma(a)\subseteq U$,
such that $a^{\dag}_{e,f}=f(a)$;

\item[\rm (xix)] $(a^{\sharp})^ka^{\dag}_{e,f}a=(a^{\dag}_{e,f})^k$;

\item[\rm (xx)] $a^k=a^{\dag}_{e,f}aa^k=a^kaa^{\dag}_{e,f}$;

\item[\rm (xxi)]  $(a^{\dag}_{e,f})^k=(a^{\sharp})^k$;

\item[\rm (xxii)] $(a^{\sharp})^k a^{\dag}_{e,f}=a^{\dag}_{e,f}(a^{\sharp})^k $;

\item[\rm (xxiii)] $a^k a^{\dag}_{e,f}=a^{\dag}_{e,f}a^k$;

\item[\rm (xxiv)] $a^{2k-1}= a^{\dag}_{e,f} a^{2k+1}a^{\dag}_{e,f}$;

\item[\rm (xxv)] $a(a^{\dag}_{e,f})^{k+1}=(a^{\sharp})^k=(a^{\dag}_{e,f})^{k+1}a$;

\item[\rm (xxvi)] $a^kaa^{\dag}_{e,f}+a^{\dag}_{e,f}aa^k=2a^k$;

\item[\rm (xxvii)] $(a^{\sharp})^{k+l-1}=(a^{\dag}_{e,f})^l (a^{\sharp})^{k-1}= (a^{\sharp})^{k-1}(a^{\dag}_{e,f})^l$;

\item[\rm (xxviii)] $aa^{\dag}_{e,f}(a^{\sharp})^{k+l-1}=(a^{\dag}_{e,f})^k(a^{\sharp})^la=(a^{\sharp})^la(a^{\dag}_{e,f})^k$,
\markright{\hskip4truecm WEIGHTED MOORE-PENROSE INVERSE } 

\item[\rm (xxix)] $a$ is weighted  EP with weights $f$ and $e$;

\item[\rm (xxx)] $a$ is both weighted EP with weights  $e$ and $e$
and with weights  $f$ and $f$;

\item[\rm (xxxi)] $a^k$ is weighted EP with weights $e$ and $f$;

\item[\rm (xxxii)] $a^{\sharp}$ is weighted EP with weights $e$
and $f$;

\item[\rm (xxxiii)]
$aa^{\sharp}=aa^\dagger_{e,e}=aa^\dagger_{f,f}$ {\rm (}or
$aa^{\sharp}=a^\dagger_{e,e}a=a^\dagger_{f,f}a${\rm )};

\item[\rm (xxxiv)] $aa^{\sharp}=aa^\dagger_{e,f}=aa^\dagger_{f,e}$
{\rm (}or $aa^{\sharp}=a^\dagger_{f,e}a=a^\dagger_{e,f}a${\rm )};

\item[\rm (xxxv)] $aa^\dagger_{e,f}(a^k+\lambda
a^\dagger_{e,f})=(a^k+\lambda a^\dagger_{e,f})aa^\dagger_{e,f}$;

\item[\rm (xxxvi)] $a^\dagger_{e,f}a(a^k+\lambda
a^\dagger_{e,f})=(a^k+\lambda a^\dagger_{e,f})a^\dagger_{e,f}a$;

\item[\rm (xxxvii)] $(a+\lambda a^\dagger_{e,e})A=(a+\lambda
a^\dagger_{f,f})A=(\lambda a+a^3)A$  and  

\item [\rm  ] $(a+\lambda
a^\dagger_{e,e})^{-1}(0)=(a+\lambda a^\dagger_{f,f})^{-1}(0)=(\lambda
a+a^3)^{-1}(0)$;

\item[\rm (xxxviii)] $(a+\lambda a^\dagger_{e,f})A=(\lambda
a+a^3)A$ and  $(a+\lambda a^\dagger_{e,f})^{-1}(0)=(\lambda
a+a^3)^{-1}(0)$,

 \end{itemize}
where $k$, $l\in \mathbb N$ and $\lambda\in \mathbb
C\setminus\{0\}$.
\end{thm}
\begin{proof}Let $a\in A$ that satisfies the hypothesis of the Theorem.
It is not difficult to prove that if $a$ is weighted EP with
weights $e$ and $f$, then  statements (i)-(xiii) and (xix)-(xxxviii)
hold. To prove the converse in this case, consider $L_a\colon A\to
A$ and apply the left multiplication representation to the
statements. According Theorem \ref{thm8} and Remark
\ref{rema19}(i), each statement is relaborated with $L_a$,
$(L_a)^{\sharp}$, and $(L_a)^{\dag}_{L_e,L_f}$ instead of $a$,
$a^{\sharp}$ and $a^{\dag}_{e,f}$ respectively. Then apply
Theorems \ref{thm20} and \ref{thm21} to prove that $L_a\in L(A)$
is weighted EP with weights $L_e$ and $L_f$ (recall that according
to Lemma \ref{lem9}, $L_e$, $L_f\in L(A)$ are invertible and
positive). Finally, apply again  Theorem \ref{thm8} to prove that
$a$ is weighted EP with weights $e$ and $f$.\par
  \markboth{  }{ \hskip1.5truecm \rm ENRICO BOASSO, DRAGAN S. DJORDJEVI\'C AND  DIJANA  MOSI\'C }
\indent To prove the equivalence between the condition of being weighted EP and  statement (xiv), note that if $aa^{\dag}_{e,f}=a^{\dag}_{e,f}a$, then
it is not difficlt to prove that
$a=(a^2+1-a^{\dag}_{e,f}a)a^{\dag}_{e,f}=a^{\dag}_{e,f}(a^2+1-aa^{\dag}_{e,f})$ and
$(a^2+1-a^{\dag}_{e,f}a)^{-1}=(a^{\dag}_{e,f})^2+1-a^{\dag}_{e,f}a$. Clearly,
statement (xiv) implies statement (xv), which in turn implies statement (xvi).
On the other hand, statement (xvi) implies statement (vi).\par

\indent Now consider statement (xvii).  If $a^{\dag}_{e,f}$ exists, then since the group inverse is unique, $a^{\dag}_{e,f}=a^{\sharp}$.
Then, according to \cite[Lemma 1.4.5]{DR}, statement (xvii) holds. If, on the
other hand $a\in A$ satisfies the condition of statement (xvii), then applying this condition to $b=a$,
$a$ is weighted EP with weights $e$ and $f$.\par

\indent Finally, to prove the equivalence between statements (i) and (xviii), suppose that  $a^{\dag}_{e,f}$ exists. Then, as in the previous paragraph,
$a^{\dag}_{e,f}=a^{\sharp}$. Then, according to \cite[Theorem 4.4]{Ko}, statement (xviii) holds.
On the other hand, if $a^{\dag}_{e,f}=f(a)$, where $f\colon U\to C$ is holomorphic ($\sigma (a)\subset U)$, then
according to \cite[Proposition 4.9, Chapter VII]{C}, $a$ is weighted EP with weights $e$ and $f$.
\end{proof}

   \markboth{  }{ \hskip1.5truecm \rm ENRICO BOASSO, DRAGAN S. DJORDJEVI\'C AND  DIJANA  MOSI\'C }

\noindent ENRICO BOASSO\par
\noindent E-mail address: enrico\_odisseo@yahoo.it
\vskip.2truecm
\noindent DRAGAN S. DJORDJEVI\'C\par
\noindent E-mail address:  dragan@pmf.ni.ac.rs 
\vskip.2truecm
\noindent DIJANA  MOSI\'C \par
\noindent E-mail address: dijana@pmf.ni.ac.rs

\end{document}